\documentclass[11pt,a4paper]{article}

\usepackage[T1]{fontenc}  
\usepackage[left=20mm, top=20mm, bottom=20mm, right=20mm]{geometry} 
\usepackage{amsmath, amsthm, amssymb, bm, mathtools}  
\usepackage{mathrsfs} 
\usepackage{cases}    

\usepackage{graphicx, xcolor}  
\usepackage{array}  
\usepackage[inline]{enumitem}  
\setlist[enumerate,1]{label=(\arabic*)}  
\usepackage{caption, subcaption}  
\usepackage[normalem]{ulem} 
\usepackage[pdftex,
    bookmarks=true,           
    bookmarksnumbered=true,   
    bookmarksopen=true,       
    plainpages=false,         
    pdfpagelabels,            
    colorlinks=true,          
    linkcolor=blue,           
    citecolor=blue,           
    anchorcolor=green,        
    urlcolor=black,           
    breaklinks=true,          
    hyperindex=true]{hyperref}

\usepackage{cleveref}

\makeatletter
\renewenvironment{proof}[1][\proofname]{%
    \par\pushQED{\qed}%
    \normalfont\topsep6\p@\@plus6\p@\relax
    \trivlist
    \item[\hskip\labelsep
          \bfseries
          \textcolor{black}{#1\@addpunct{.}}]
    \ignorespaces
}{%
    \popQED\endtrivlist\@endpefalse
}
\makeatother

\newtheorem{theorem}{Theorem}
\newtheorem{lem}{Lemma}
\newtheorem{cor}{Corollary}
\newtheorem{conjecture}{Conjecture}
\theoremstyle{definition}
\newtheorem{definition}{Definition}
\newtheorem{proposition}{Proposition}

\newtheorem{obs}{Observation}

\usepackage{tikz} 
\usetikzlibrary{calc, shapes, decorations.markings, arrows, patterns, snakes}

\pgfdeclarelayer{nodelayer}   
\pgfdeclarelayer{edgelayer}   
\pgfsetlayers{main, nodelayer, edgelayer} 

\tikzset{
    blue dot/.style={ 
        circle,
        fill=blue,
        minimum size=5pt,
        inner sep=0pt,
        outer sep=0pt
    },
    arc 20/.style={ 
        line width=0.6pt,
        black
    },
    label text/.style={ 
        font=\small,
        inner sep=1pt
    }
}

\usepackage[square, numbers, sort&compress]{natbib}  

\usepackage{soul}  

\definecolor{LightBlue}{RGB}{173,216,230}

\usepackage[colorinlistoftodos,prependcaption,textsize=scriptsize, textwidth=20mm,]{todonotes}

\begin{document}
	\title{Tighter Bounds on the Degree-Truncated Choice Number of Planar Graphs}
	\author{Huijuan Xu\thanks{School of Mathematical Sciences, Zhejiang Normal University, Email: xuhuijuan@zjnu.edu.cn.} \and Huan Zhou\thanks{School of Mathematical Sciences, Zhejiang Normal University, Email: huanzhou@zjnu.edu.cn.} \and Jialu Zhu \thanks{Data Science Institute, Shandong University, Jinan, China. Email: jialuzhu@zjnu.edu.cn.} \and        Xuding Zhu\thanks{School of Mathematical Sciences, Zhejiang Normal University, Email: xdzhu@zjnu.edu.cn, Grant Numbers: NSFC 12371359.} }
\date{
    \today
}
 
	\maketitle
	
	\begin{abstract}
	Assume $G$ is a graph and $k$ is a positive integer. Let $f:V(G)\to \mathbb{N}$ be defined as $f(v)=\min\{k,d_G(v)\}$. If $G$ is $f$-choosable, then we say $G$ is degree-truncated $k$-choosable. The degree-truncated choice number of $G$ is $\operatorname{ch}^{\text{\st{d}}}(G) = \min\{k: G \text{ is degree-truncated $k$-choosable}\}$. For a family $\mathcal{G}$ of graphs, $\operatorname{ch}^{\text{\st{d}}}(\mathcal{G}) = \max\{\operatorname{ch}^{\text{\st{d}}}(G):G \in \mathcal{G}\}$. Let $\mathcal{P}$ denote the family of 3-connected non-complete planar  graphs. Richter asked in 2008 whether $ch^{\text{\st{d}}}(\mathcal{P}) \le 6$.  In 2025, Zhou, Zhu and Zhu answered this question in negative and proved that $8 \le ch^{\text{\st{d}}}(\mathcal{P}) \le 16$.  This result was improved by 
    Jiang, Xu, Xu, and Zhu, who proved that 
    $9 \le ch^{\text{\st{d}}}(\mathcal{P}) \le 12$. In this paper, 
     we further improve the result and  prove that 
     $10 \le \operatorname{ch}^{\text{\st{d}}}(\mathcal{P}) \le 11$. We conjecture that $\operatorname{ch}^{\text{\st{d}}}(\mathcal{P}) =10$, and we confirm this conjecture for those planar graphs $G \in \mathcal{P}$ for which the subgraph induced by vertices of degree at least 11 is 4-choosable.
 
\end{abstract}
\section{Introduction}

A \textit{list assignment} $L$ for a graph $G$ is a mapping that assigns a set $L(v)$ of permissible colours to each vertex $v$ of $G$.
We denote by $\mathbb{N}^G$   the set of mappings from $V(G)$ to $\mathbb{N}={\left\{0,1, \ldots \right\} }  $. 
For $f \in \mathbb{N}^G$, an $f$-assignment is a list assignment $L$ with $|L(v)| \geq f(v)$ for each vertex $v$.
Given a list assignment $L$ of $G$, an $L$-colouring of $G$ is a mapping $\phi$ that assigns a colour $\phi(v)$ from $L(v)$ to each vertex $v$, such that $\phi(u) \ne \phi(v)$ for every edge $uv$ of $G$. We say $G$ is \emph{$f$-choosable} if $G$ is $L$-colourable for any $f$-assignment $L$ of $G$.   The \emph{choice number} $\operatorname{ch}(G)$ of $G$ is the minimum integer $k$ such that $G$ is $k$-choosable. 

List colouring of graphs was introduced by Erd\H{o}s, Rubin and Taylor \cite{ERT} and Vizing \cite{Vizing} in 1970's, and  is now an important part of chromatic graph theory. One of the main topics in this area is list colouring of planar graphs. It was conjectured in \cite{ERT} that every planar graph is 5-choosable, and not every planar graph is 4-choosable. These conjectures were confirmed by Thomassen \cite{Thomassen} and Voigt \cite{Voigt} in 1994 and 1993, respectively. 

To prove that a graph $G$ is $k$-choosable, we may assume $G$ has minimum degree at least $k$, because if  $v \in V(G)$ is a vertex of degree less than $k$, then $G$ is $k$-choosable if and only if $G-v$ is $k$-choosable. Instead of a $k$-assignment, one may consider list assignments for which vertices of small degree have less available colours.  Indeed, list colouring of graphs is often applied in inductive proofs in the study of classical colouring. To prove a graph $G$ is $k$-colourable, we may first colour some of the vertices. The colouring of the remaining vertices  is then a list colouring problem: The set of permissible colours to an uncoloured vertex $v$ consists of those colours from $[k]$ not used by any of its coloured neighbours. 
If a vertex $v$ has small degree in the subgraph induced by uncoloured vertices, then  it likely has more coloured neighbours in the original graph, and hence may have less permissible colours.

Richter considered the following problem (cf. \cite{Hutchinson}): Assume $G$ is a planar graph and $L$ is a list assignment that assigns to every vertex $v$ a set of $k$ colours, except that vertices $v$ of degree less than $k$ have $|L(v)|=d_G(v)$,  where $d_G(v)$ is the degree of $v$.
Will $G$ be $L$-colourable? 

For the question to be meaningful, some graphs need to be excluded. 
A graph  $G$ is called  {\em  degree-choosable} if it is $f$-choosable for the mapping $f \in \mathbb{N}^G$ defined as $f(v)=d_G(v)$ for $v \in V(G)$. 
Degree-choosable graphs are characterized in \cite{ERT} and \cite{Vizing}: A connected graph $G$ is not degree-choosable if and only if $G$ is a Gallai-tree, i.e.,  each block of $G$ is either a complete graph or an odd cycle. If $G$ is not degree-choosable, then for any positive integer $k$,  the question above has a negative answer. So Gallai-trees need to be excluded. For the graph $K_{2, k^2}$,   which is planar and not a Gallai-tree, the answer is also easily seen to be negative. To exclude these obvious counterexamples, Richter asked the following question:

If $G$ is a $3$-connected non-complete planar graph, is it true that $G$ is $f$-choosable, where $f(v)=\min\{6, d_G(v)\}$?

The question are studied in a few papers. The following definition was given in \cite{ZZZ25}.

\begin{definition}
    \label{def-dtcn}
    Let $G$ be a graph and $k$ be a positive integer. We say $G$ is {\em degree-truncated $k$-choosable} if $G$ is $f$-choosable for the function $f \in \mathbb{N}^G$ defined as $f(v)=\min\{k, d_G(v)\}$ for each vertex $v$. The {\em degree-truncated choice number} $\operatorname{ch}^{\text{\st{d}}}(G)$ of $G$ is the minimum integer $k$ such that $G$ is degree-truncated $k$-choosable. For a family $\mathcal{G}$ of graphs, let 
    $\operatorname{ch}^{\text{\st{d}}}(\mathcal{G})= \max \{\operatorname{ch}^{\text{\st{d}}}(G): G \in \mathcal{G}\}.$ 
\end{definition}
We denote by $\mathcal{P}$ the family of 3-connected non-complete planar graphs. 
Richter's question is whether $\operatorname{ch}^{\text{\st{d}}}(\mathcal{P}) \le 6$?
 
Motivated by Richter's question, 
 Hutchinson \cite{Hutchinson} studied the degree-truncated choice number of outerplanar graphs. She proved that if $G \ne K_3$ is a 2-connected maximal outerplanar graph (i.e., each inner face is a triangle), then $G$ is degree-truncated 5-choosable. If $G$ is a 2-connected  bipartite outerplanar graph, then $G$ is degree-truncated 4-choosable. The results are sharp. Kostochka constructed a 2-connected maximal outerplanar graph that is not degree-truncated 4-choosable, and Hutchinson constructed a 2-connected bipartite outerplanar graph  that is not degree-truncated 3-choosable. 

 In \cite{CPTV},   some families of $K_5$-minor free graphs are shown to be degree-truncated $k$-choosable for $k= 6,7,8$, by putting restrictions on the distance between components of the subgraph induced by vertices of degree less than $k$. 

 Hutchinson asked whether every 2-connected outerplanar graph other than odd cycles are also degree-truncated $5$-choosable. This question was answered in affirmative and the conclusion are strengthened and generalized in \cite{LWZZ} and \cite{DZ}.  DP-colouring of a graph is a variation of list colouring of a graph introduced in \cite{DP}. We  refer the reader to \cite{DP} for its definition. For our purpose, we just remark that if $G$ is DP-$f$-colourable, then $G$ is $f$-choosable. We say $G$ is {\em degree-truncated DP-$k$-colourable}, if $G$ is DP-$f$-colourable for $f$ defined as $f(v)=\min\{k, d_G(v)\}$.  The {\em degree-truncated DP-chromatic number} 
 $\chi_{DP}^{\text{\st{d}}}(G)$ of $G$ is the minimum $k$ such that $G$ is degree-truncated DP-$k$-colourable. Thus $\chi_{DP}^{\text{\st{d}}}(G) \ge \operatorname{ch}^{\text{\st{d}}}(G)$ for any graph $G$, and it is known that the difference $\chi_{DP}^{\text{\st{d}}}(G) - \operatorname{ch}^{\text{\st{d}}}(G)$ can be arbitrarily large. 
 It was proved in \cite{LWZZ} that 2-connected $K_{2,4}$-minor free graphs other than cycles and complete graphs are degree-truncated DP-$5$-colourable. Note that a graph $G$ is outerplanar if and only if $G$ is $K_{2,3}$-minor free  and $K_4$-minor free. So the family of $K_{2,4}$-minor free graphs is a larger family of graphs, and the conclusion is stronger than degree-truncated 5-choosable. 

 An   orientation $D$ of a graph $G$ is 
 called an {\em Alon-Tarsi orientation} (AT-orientation, for short) if $|\mathcal{E}_e(D)|=|\mathcal{E}_o(D)|$, where   $\mathcal{E}_e(D)$ (respectively, $\mathcal{E}_o(D)$) is the family of spanning Eulerian subdigraphs $H$ of $D$ ($H$ is Eulerian if $d_H^+(v) = d_H^-(v)$ for each vertex $v$)  with an even number of arcs (respectively, an odd number of arcs). We say $G$ is $f$-AT if $G$ has an AT-orientation $D$ with $d_D^+(v) < f(v)$ for each vertex $v$. It was proved in \cite{AT} that if $G$ is $f$-AT, then $G$ is $f$-choosable. 
 We say $G$ is {\em degree-truncated $k$-AT} if $G$ is $f$-AT for the function defined as $f(v)=\min\{k, d_G(v)\}$ for each vertex $v$. The {\em degree-truncated AT number} 
 ${\rm AT}^{\text{\st{d}}}(G)$ of $G$ is the minimum $k$ such that $G$ is degree-truncated $k$-AT. Thus ${\rm AT}^{\text{\st{d}}}(G) \ge \operatorname{ch}^{\text{\st{d}}}(G)$ for any graph $G$, and it is also known that the difference ${\rm AT}^{\text{\st{d}}}(G) - \operatorname{ch}^{\text{\st{d}}}(G)$ can be arbitrarily large. 
 It was   proved in \cite{DZ} that 2-connected outerplanar graphs other than odd cycles are degree-truncated 5-AT. 
 
Richter's question was answered in negative in
 \cite{ZZZ25}. Along with Richter's question,  it also remained as an open problem for more than a decade whether $\operatorname{ch}^{\text{\st{d}}}(\mathcal{P})$ is bounded by any constant. This question was also settled in \cite{ZZZ25}, where it was 
  proved that  $8 \le \operatorname{ch}^{\text{\st{d}}}(\mathcal{P}) \le 16$. 

This result motivates a challenging open problem: What is the exact value of $\operatorname{ch}^{\text{\st{d}}}(\mathcal{P})$?

The lower and upper bounds for $\operatorname{ch}^{\text{\st{d}}}(\mathcal{P})$ was improved in \cite{JXXZ}, where it was prove
that $9 \le \operatorname{ch}^{\text{\st{d}}}(\mathcal{P}) \le 12$.
In this paper, we further improve this result and prove that $10 \le \operatorname{ch}^{\text{\st{d}}}(\mathcal{P}) \le 11$. We propose the following conjecture.

\begin{conjecture}
    \label{conj-1}
Every 3-connected non-complete planar graph is degree-truncated $10$-choosable. Consequently,     $\operatorname{ch}^{\text{\st{d}}}(\mathcal{P})=10$.
\end{conjecture}

Our result shows that  Conjecture \ref{conj-1} holds for those planar graphs $G$ for which $G[X]$ is 4-choosable, where $X=\{v:d_G(v) \ge 11\}$.

\section{The lower bound}

In this section, we present a 3-connected non-complete planar graph that is not degree-truncated 9-choosable.
Consequently, $\operatorname{ch}^{\text{\st{d}}}(\mathcal{P})
\ge 10$.

The graph is constructed in a few steps. Each step construct a graph with a list assignment, and these  graphs and lists are combined together to obtain the final graph and list. We use two operations to combine graphs and lists together. 

\begin{obs}
\label{obs-1}
The following is easy and the proof is omitted.
\begin{itemize}
    \item For $i=1,2$, let $G_i$ be a graph, $v_i$ be a vertex of $G_i$, and $L_i$ be an assignment of $G_i$ such that $G_i$ is not $L_i$-colourable.
    Let $G$ be obtained from $G_1$ and $G_2$ by taking their disjoint union, and identify $v_1$ and $v_2$ into a single vertex $v^*$. Then $L=L_1 \cup L_2$ (i.e., $L(v) = L_i(v)$ if $v \in V(G_i)-\{v_i\}$ and $L(v^*) =L_1(v_1) \cup L_2(v_2)$) is an assignment of $G$ and $G$ is not $L$-colourable.
    \item Assume $G$ is a graph, $L$ is a list assignment of $G$ such that $G$ is not $L$-colourable. Let $G'$ be obtained from $G$ by adding one vertex $v$ adjacent to some vertices of $G$, and $L'$ be obtained from $L$ by letting $L'(v)=\{a\}$ and $L'(x) = L(x) \cup \{a\}$ for each neighbour $x$ of $v$, and $L'(x)=L(x)$ otherwise. Then $G'$ is not $L'$-colourable.
\end{itemize}
 \end{obs}


\begin{figure}[h]
    \centering
    \begin{tikzpicture}[
    every node/.style={font=\small},
    base/.style={circle,draw,minimum size=6mm,inner sep=0pt},
    minor/.style={circle,draw,minimum size=2mm,inner sep=0pt,fill=black},
    matline/.style={thick},
    crossedline/.style={thick,dashed,red}
  ]

\node[base] (y1) at (10,4)   {$x_2$};
\node[base] (x1) at (0,4)   {$x_1$};

\draw[thick] (x1) -- (y1);


\node[minor] (z1) at (2.2,5.4) {};
\node[minor] (z2) at (3.5,5.5) {};
\node[minor] (z3) at (2,6) {};

\node[base] (z) at (5,6)   {$z$};

\draw[thick] (x1) -- (z1);
\draw[thick] (x1) -- (z2);
\draw[thick] (x1) -- (z3);
\draw[thick] (z1) -- (z2) -- (z3) -- (z1);
\draw[thick] (z2) -- (z);

\node[minor] (w1) at (10-2.2,5.4) {};
\node[minor] (w2) at (10-3.5,5.5) {};
\node[minor] (w3) at (10-2,6) {};

\draw[thick] (y1) -- (w1);
\draw[thick] (y1) -- (w2);
\draw[thick] (y1) -- (w3);
\draw[thick] (w1) -- (w2) -- (w3) -- (w1);
\draw[thick] (w2) -- (z);


\node[minor] (z1b) at (2.2,8-5.4) {};
\node[minor] (z2b) at (3.5,8-5.5) {};
\node[minor] (z3b) at (2,8-6) {};

\node[base] (zb) at (5,8-6)   {$z'$};

\draw[thick] (x1) -- (z1b);
\draw[thick] (x1) -- (z2b);
\draw[thick] (x1) -- (z3b);
\draw[thick] (z1b) -- (z2b) -- (z3b) -- (z1b);
\draw[thick] (z2b) -- (zb);

\node[minor] (w1b) at (10-2.2,8-5.4) {};
\node[minor] (w2b) at (10-3.5,8-5.5) {};
\node[minor] (w3b) at (10-2,8-6) {};

\draw[thick] (y1) -- (w1b);
\draw[thick] (y1) -- (w2b);
\draw[thick] (y1) -- (w3b);
\draw[thick] (w1b) -- (w2b) -- (w3b) -- (w1b);
\draw[thick] (w2b) -- (zb);


\node  at (10,3.5)   {$1234$};
\node  at (0,3.5)   {$1234$};

\node  at (2.5,5.1) {234};
\node  at (4,5.2) {1234};
\node  at (2.5,6.3) {234};

\node  at (5,5.5) {12};

\node  at (10-2.2,5.1) {134};
\node  at (10-3.5,5.2) {1234};
\node  at (10-2,6.3) {134};


\node  at (2.5,8-5.1) {134};
\node  at (4,8-5.2) {1234};
\node  at (2.5,8-6.3) {134};

\node  at (5,8-5.5) {12};

\node  at (10-2.2,8-5.1) {234};
\node  at (10-3.5,8-5.2) {1234};
\node  at (10-2,8-6.3) {234};


\end{tikzpicture}

 \caption{The graph $H_0$ and list assignment $L_0$}
    \label{fig-H0}
\end{figure}
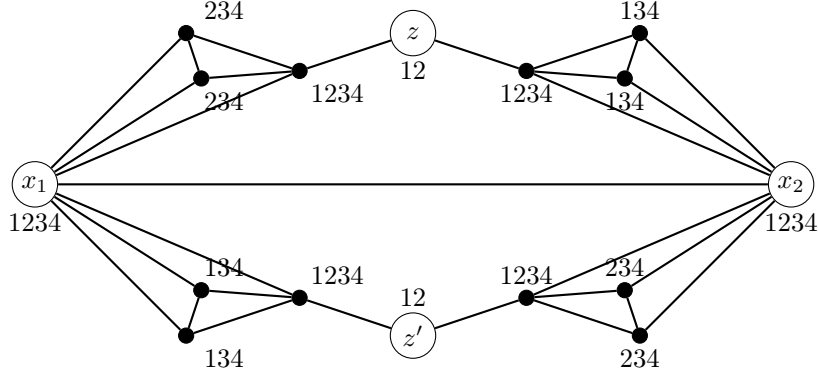

Let $H_0$ and $L_0$ be the graph and list assignment $L_0$ given in Figure \ref{fig-H0}. First we show that $H_0$ is not $L_0$-colourable.

Assume to the contrary that $\phi$ is an $L_0$-colouring of $H_0$. If $z$ is coloured by $2$, then the top right triangle will use all the colours $\{1,3,4\}$. Hence $\phi(x_2)=2$. Thus the copy of $K_4$ containing $x_2$ uses all colours $\{1,2,3,4\}$. This  forces the vertex of this $K_4$ adjacent to $z'$  be coloured by $1$, and hence 
$z'$ is coloured by $2$. Hence the bottom left triangle uses all the colours $\{1,3,4\}$ and hence $\phi(x_1)=2$, which is a contradiction (as $\phi(x_2)=2$). 

If $z$ is coloured by $1$, then the top left triangle will use all the colours $\{2,3,4\}$. Hence $\phi(x_1)=1$. Thus the copy of $K_4$ containing $x_1$ uses all colours $\{1,2,3,4\}$. This  forces the vertex of this $K_4$ adjacent to $z'$  be coloured by $2$, and hence 
$z'$ is coloured by $1$. Hence the bottom right triangle uses all the colours $\{2,3,4\}$ and hence $\phi(x_2)=1$, which is a contradiction (as $\phi(x_1)=1$).

  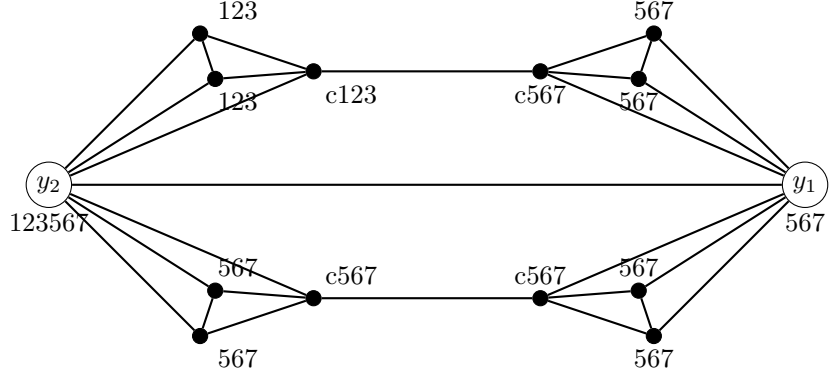
\begin{figure}[!htb]
    \centering
    \begin{tikzpicture}[every node/.style={font=\small},
    base/.style={circle,draw,minimum size=6mm,inner sep=0pt},
    minor/.style={circle,draw,minimum size=2mm,inner sep=0pt,fill=black},
    matline/.style={thick},
    crossedline/.style={thick,dashed,red}
  ]

\node[base] (y1) at (10,4)   {$y_1$};
\node[base] (x1) at (0,4)   {$y_2$};

\node  at (10,3.5)   {$567$};
\node  at (0,3.5)   {$123567$};

\draw[thick] (x1) -- (y1);


\node[minor] (z1) at (2.2,5.4) {};
\node[minor] (z2) at (3.5,5.5) {};
\node[minor] (z3) at (2,6) {};

\node  at (2.5,5.1) {123};
\node  at (4,5.2) {c123};
\node  at (2.5,6.3) {123};

\draw[thick] (x1) -- (z1);
\draw[thick] (x1) -- (z2);
\draw[thick] (x1) -- (z3);
\draw[thick] (z1) -- (z2) -- (z3) -- (z1);

\node[minor] (w1) at (10-2.2,5.4) {};
\node[minor] (w2) at (10-3.5,5.5) {};
\node[minor] (w3) at (10-2,6) {};

\node  at (10-2.2,5.1) {567};
\node  at (10-3.5,5.2) {c567};
\node  at (10-2,6.3) {567};

\draw[thick] (y1) -- (w1);
\draw[thick] (y1) -- (w2);
\draw[thick] (y1) -- (w3);
\draw[thick] (w1) -- (w2) -- (w3) -- (w1);
\draw[thick] (w2) -- (z2);


\node[minor] (z1b) at (2.2,8-5.4) {};
\node[minor] (z2b) at (3.5,8-5.5) {};
\node[minor] (z3b) at (2,8-6) {};


\draw[thick] (x1) -- (z1b);
\draw[thick] (x1) -- (z2b);
\draw[thick] (x1) -- (z3b);
\draw[thick] (z1b) -- (z2b) -- (z3b) -- (z1b);

\node[minor] (w1b) at (10-2.2,8-5.4) {};
\node[minor] (w2b) at (10-3.5,8-5.5) {};
\node[minor] (w3b) at (10-2,8-6) {};

\draw[thick] (y1) -- (w1b);
\draw[thick] (y1) -- (w2b);
\draw[thick] (y1) -- (w3b);
\draw[thick] (w1b) -- (w2b) -- (w3b) -- (w1b);
\draw[thick] (w2b) -- (z2b);


\node  at (10-2.2,8-5.1) {567};
\node  at (10-3.5,8-5.2) {c567};
\node  at (10-2,8-6.3) {567};


\node  at (2.5,8-5.1) {567};
\node  at (4,8-5.2) {c567};
\node  at (2.5,8-6.3) {567};

\end{tikzpicture}
\caption{The graph $H_1$ and list assignment $L_1$}
    \label{fig-H1}
\end{figure}

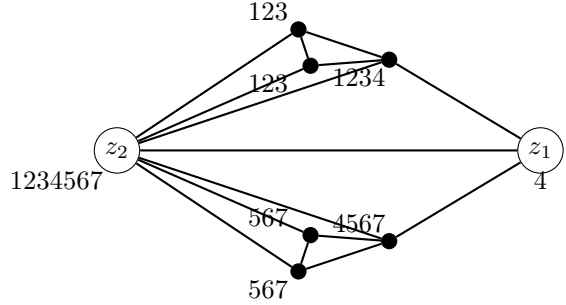
\begin{figure}[!htb]
\centering
    \begin{tikzpicture}[scale=0.8, every node/.style={font=\small},
    base/.style={circle,draw,minimum size=6mm,inner sep=0pt},
    minor/.style={circle,draw,minimum size=2mm,inner sep=0pt,fill=black},
    matline/.style={thick},
    crossedline/.style={thick,dashed,red}
  ]

\node[base] (x2) at (7,4)   {$z_1$};
\node  at (7,3.5) {4};

\node[base] (x3) at (0,4)   {$z_2$};
\node  at (-1,3.5)   {$1234567$};

\draw[thick] (x2) -- (x3);


\node[minor] (z1) at (1+2.2,5.4) {};
\node[minor] (z2) at (1+3.5,5.5) {};
\node[minor] (z3) at (1+2,6) {};

\node  at (2.5,5.1) {123};
\node  at (4,5.2) {1234};
\node  at (2.5,6.3) {123};

\draw[thick] (x3) -- (z1);
\draw[thick] (x3) -- (z2);
\draw[thick] (x3) -- (z3);
\draw[thick] (z1) -- (z2) -- (z3) -- (z1);
\draw[thick] (z2) -- (x2);


\node[minor] (z1b) at (1+2.2,8-5.4) {};
\node[minor] (z2b) at (1+3.5,8-5.5) {};
\node[minor] (z3b) at (1+2,8-6) {};


\draw[thick] (x3) -- (z1b);
\draw[thick] (x3) -- (z2b);
\draw[thick] (x3) -- (z3b);
\draw[thick] (z1b) -- (z2b) -- (z3b) -- (z1b);
\draw[thick] (z2b) -- (x2);


\node  at (2.5,8-5.1) {567};
\node  at (4,8-5.2) {4567};
\node  at (2.5,8-6.3) {567};

\end{tikzpicture}
\caption { $H_2$  and list assignment $L_2$}
\label{fig-H2}
\end{figure}

Let graph $H_1$ and list assignment $L_1$, graph $H_2$ and list assignment $L_2$ be given above. It is straightforward to verify that $H_i$ is not $L_i$-colourable. The verification is similar to the case for $H_0$ and $L_0$ and is omitted. 

Now we take the disjoint union of a copy of $H_0$, two copies of $H_1$ and two copies of $H_2$, combine them as indicated in Figure \ref{fig-HH} below. (To be precise, let the two named vertices of the two copies of $H_1$ be $y_1^1, y_2^1$ and $y_1^2,y_2^2$, respectively,  the two named vertices of the two copies of $H_2$ be $z_1^1, z_2^1$ and $z_1^2,z_2^2$, respectively,
the two named vertices of $H_0$ by $x_1^0, x_2^0$, respectively. Then in Figure \ref{fig-HH}, $x_1$ is the identification of 
$y_1^1$ and $ x_1^0$,  $y_1$ is the identification of $y_1^2$ and $x_2^0$, $x_2$ is the identification of 
$y_2^1$ and $z_1^1$,  $y_2$ is the identification of 
$y_2^2$ and $z_1^2$, $x_3=z_2^1, y_3=z_2^2$.) 
Let $L'$ be the union of the corresponding list assignments.

We obtain a plane graph $H'$ which is not $L'$-colourable. It is easy to see from the construction that $L'(x_i)=L'(y_i)=\{1,2,3,4,5,6,7\}$, and for every other vertex $v$ of $H'$, $|L'(v)| = d_{H'}(v)$.

\begin{figure}[!htb]
\centering 
    \begin{tikzpicture}[scale=0.6, every node/.style={font=\small},
    base/.style={circle,draw,minimum size=6mm,inner sep=0pt},
    minor/.style={circle,draw,minimum size=2mm,inner sep=0pt,fill=black},
    matline/.style={thick},
    crossedline/.style={thick,dashed,red}
  ]

\node[base] (xx1) at (5.3+0,4-6)   {$x_1$};

\node[base] (xx2) at (-3+5.3+0,4-6)   {$x_2$};

\node[base] (xx3) at (-6+5.3+0,4-6)   {$x_3$};

\node[base] (yy1) at (5.3+3,4-6)   {$y_1$};

\node[base] (yy2) at (3+5.3+3,4-6)   {$y_2$};

\node[base] (yy3) at (6+5.3+3,4-6)   {$y_3$};

\draw[thick]  (xx1) ..controls (6.8, -1)..  (yy1);
\draw[thick]  (xx1) ..controls (6.8, -3)..  (yy1);

\draw[thick]  (xx1) ..controls (6.8-3, -1)..  (xx2);
\draw[thick]  (xx1) ..controls (6.8-3, -3)..  (xx2);

\draw[thick]  (xx2) ..controls (6.8-6, -1)..  (xx3);
\draw[thick]  (xx2) ..controls (6.8-6, -3)..  (xx3);

\draw[thick]  (yy1) ..controls (6.8+3, -1)..  (yy2);
\draw[thick]  (yy1) ..controls (6.8+3, -3)..  (yy2);

\draw[thick]  (yy2) ..controls (6.8+6, -1)..  (yy3);
\draw[thick]  (yy2) ..controls (6.8+6, -3)..  (yy3);

\node at (5.3+1.5,4-6) {$H_0$};

\node at (-3+5.3+1.5,4-6) {$H_1$};

\node at (3+5.3+1.5,4-6) {$H_1$};

\node at (-6+5.3+1.5,4-6) {$H_2$};

\node at (6+5.3+1.5,4-6) {$H_2$};

\node at (5.3+1.5,4-9) {$L'(x_i)=L'(y_i)=\{1234567\}$.};

\end{tikzpicture}
\caption{ $H'$ is not $L'$-colourable}
\label{fig-HH}
\end{figure}

The boundary of $H'$ consists of two paths  from $x_3$ to $y_3$: $P_1$ on the top and $P_2$ on the bottom (with
$V(P_1) \cap V(P_2) = \{x_1, x_2,x_3, y_1, y_2, y_3 \} $).
Next we add two vertices $u$ and $v$ to $H'$ as indicated in Figure \ref{fig-gadget}, where $v$ is adjacent to every vertex on $P_1$, and $u$ is adjacent to every vertex on $P_2$ 
(not every edge incident to $v$ (or $u$) is drawn in the figure). 

Let $L(u)=\{a\}, L(v)=\{b\}$, and add colour $a$ to the lists of neighbours of $u$, and $b$ to the lists of neighbours of $v$. By Observation \ref{obs-1}, the resulting graph $H$ is not $L$-colourable. Note that 
$L(x_i)=L(y_i)=\{a,b,1,2,3,4,5,6,7\}$ and for every other vertex $x$ of $H$, $|L(x)|=d_H(x)$.

\begin{figure}[!htb]
\centering 
    \begin{tikzpicture}[scale=0.6, every node/.style={font=\small},
    base/.style={circle,draw,minimum size=6mm,inner sep=0pt},
    minor/.style={circle,draw,minimum size=2mm,inner sep=0pt,fill=black},
    matline/.style={thick},
    crossedline/.style={thick,dashed,red}
  ]

\node[base] (v) at (6.8, 0) {$v$};
\node[base] (u) at (6.8, -4) {$u$};

\draw[thick] (u) -- (xx1);
\draw[thick] (u) ..controls (6.8-3, -3.5).. (xx2);
\draw[thick] (u) ..controls (6.8-6, -3.5).. (xx3);
\draw[thick] (u) -- (yy1);
\draw[thick] (u) ..controls (6.8+3, -3.5).. (yy2);
\draw[thick] (u) ..controls (6.8+6, -3.5).. (yy3);

\draw[thick] (v) -- (xx1);
\draw[thick] (v) ..controls (6.8-3, -0.5).. (xx2);
\draw[thick] (v) ..controls (6.8-6, -0.5).. (xx3);
\draw[thick] (v) -- (yy1);
\draw[thick] (v) ..controls (6.8+3, -0.5).. (yy2);
\draw[thick] (v) ..controls (6.8+6, -0.5).. (yy3);

\node[base] (xx1) at (5.3+0,4-6)   {$x_1$};

\node[base] (xx2) at (-3+5.3+0,4-6)   {$x_2$};

\node[base] (xx3) at (-6+5.3+0,4-6)   {$x_3$};

\node[base] (yy1) at (5.3+3,4-6)   {$y_1$};

\node[base] (yy2) at (3+5.3+3,4-6)   {$y_2$};

\node[base] (yy3) at (6+5.3+3,4-6)   {$y_3$};

\draw[thick]  (xx1) ..controls (6.8, -1)..  (yy1);
\draw[thick]  (xx1) ..controls (6.8, -3)..  (yy1);

\draw[thick]  (xx1) ..controls (6.8-3, -1)..  (xx2);
\draw[thick]  (xx1) ..controls (6.8-3, -3)..  (xx2);

\draw[thick]  (xx2) ..controls (6.8-6, -1)..  (xx3);
\draw[thick]  (xx2) ..controls (6.8-6, -3)..  (xx3);

\draw[thick]  (yy1) ..controls (6.8+3, -1)..  (yy2);
\draw[thick]  (yy1) ..controls (6.8+3, -3)..  (yy2);

\draw[thick]  (yy2) ..controls (6.8+6, -1)..  (yy3);
\draw[thick]  (yy2) ..controls (6.8+6, -3)..  (yy3);

\node at (5.3+1.5,4-6) {$H_0$};

\node at (-3+5.3+1.5,4-6) {$H_1$};

\node at (3+5.3+1.5,4-6) {$H_1$};

\node at (-6+5.3+1.5,4-6) {$H_2$};

\node at (6+5.3+1.5,4-6) {$H_2$};

\end{tikzpicture}

\caption { $H$ and list assignment $L$, each of $u$ and $v$ is adjacent to every vertex "visible" to them, respectively.}
\label{fig-gadget}
\end{figure}
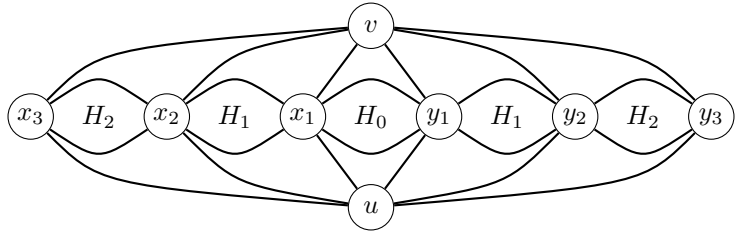

 Let $G$ be a graph obtained from the disjoint union of $72$ copies $H_i$ of $H$ by identifying all the copies of $v$ into a single vertex (also named as $v$) and all the copies of $u$ into a single vertex (also named as $u$), and   adding edges $y_3^{(i)}x_3^{(i+1)}$  (where $y_3^{(i)}$ and $x_3^{(i)}$ are the copies of $y_3$ and $x_3$ in $H_i$)  for $i=1,2,\ldots, 71$, and adding an edge connecting $v$ and $u$. Then $G$ is a non-complete planar graph. Now we show  that the graph $G$ is 3-connected. For any 2-subset $S$ of $V(G)$, if $S=\{v,u\}$, then obviously $G-S$ is connected. If $S\ne \{v,u\}$, say $v \notin S$, then distinct copies of $H$ are connected via $v$ in $G-S$. As for each copy $H_i$ of $H$, $H_i-S$ is connected, we conclude that $G-S$ is connected.
     Thus $G$ is 3-connected.

     Let $L(v)=L(u)=\{a,b,c,d,e,f,g,h,i\}$. There are $72$ possible $L$-colourings $\phi$ of $v$ and $u$. Each such a colouring $\phi$ corresponds to one copy of $H$. We define the list assignment of the corresponding  copy of $H$ as $L$ by replacing $a$ with $\phi(u)$ and replacing $b$ with $\phi(v)$ (since $vu$ is an edge of $G$, $\phi(v) \ne \phi(u)$).
     It is easy to verify that $|L(x)|= \min\{d_G(x),9\}$ for any $x\in V(G)$. As every possible $L$-colouring of $v$ and $u$ cannot be extended to an $L$-colouring of some copy of $H$, we conclude that $G$ is not $L$-colourable. Hence $G$ is not degree-truncated $9$-choosable.

\section{The upper bound} 

This section proves the following result.

\begin{theorem}
    \label{thm-main}
    Every 3-connected non-complete planar graph is degree-truncated 11-choosable, i.e., $\operatorname{ch}^{\text{\st{d}}}(\mathcal{P}) \le 11$. 
\end{theorem}

For the purpose of using induction, we prove a stronger and more technical result.

Assume $G$ is a plane graph. Let  $V_1=\{v\in V(G): d_G(v) \le 10\}$ and $V_2=V(G)-V_1$. By adding edges between vertices in $V_2$, we may assume that each face $\theta$ of $G[V_2]$ contains at most one connected component of $G[V_1]$. 
    For a connected component $Q$ of $G[V_1]$, let $\theta_Q$ be the face of $G[V_2]$ that contains $Q$.

For a cycle $C$ in $G $, denote by ${\rm int}(C)$ (respectively, ${\rm ext}(C)$) the set of 
vertices of $G$ contained in   the finite region  (respectively, the infinite region) with boundary $C$.

\begin{definition}\label{def-connected}
  We say a connected component $Q$ of $G[V_1]$  is {\em properly connected to $V_2$} if the following hold:
\begin{enumerate}
\item[(P1)] Each vertex on the boundary of $\theta_Q$ is adjacent to some vertex of $Q$.
    \item[(P2)] If $C$ is a cycle in $G[V(\theta_Q)]$ and $V(Q) \subseteq {\rm int}(C)$, 
    then for any vertex $v$ of $Q$, $G$ has three   paths contained in $V(C) \cup {\rm int}(C)$
    connecting $V(C)$ and $v$, and these paths are vertex disjoint, except that they  share the same end vertex $v$. 
   \end{enumerate}
\end{definition}

It is easy to see that if $G$ is 3-connected, then each connected component $Q$ of $G[V_1]$ is properly connected to $V_2$ (see \cite{JXXZ} for a rigorous proof). Thus Theorem \ref{thm-main} follows from the following result.

\begin{theorem}
    \label{thm-main2}
    Assume that $G$ is a connected plane graph, and  $V_1, V_2$ is a partition of $V(G)$,  where $V_2 \ne \emptyset$ and each face of $G[V_2]$ contains at most one connected component of $G[V_1]$.   Assume that each connected component $Q$ of $G[V_1]$ is properly connected to $V_2$.  Assume that $\theta^*$ is the infinite face of $G[V_2]$ and $v^* \in V(\theta^*)$. If $f: V(G) \to \mathbb{N}$ is defined as follows:
    \[
    f(v) = \begin{cases} 
    1, &\text{ if $v = v^*$ and $v^*$ is not an isolated vertex in $G[V_2]$}, \cr
    11, &\text{ if $v \in V_2 -\{v^*\}$ or $v=v^*$ is an isolated vertex in $G[V_2]$}, \cr 
    d_G(v), &\text{ if $v \in V_1$}. 
    \end{cases} 
    \]
    then $G$ is $f$-choosable.
\end{theorem}

The remainder of this paper is devoted to the proof of Theorem \ref{thm-main2}.

Assume that the theorem is not true, and $G$ is a counterexample with minimum number of vertices, and $L$ is a list assignment of $G$ that satisfies the condition of Theorem \ref{thm-main2}, and $G$ is not $L$-colourable.

For  $X, Y \subseteq V(G)$, let $$N_X(Y)= N_G(Y) \cap X,$$
and we write $N_X(v)$ for $N_X(\{v\})$.

\begin{definition}
    Assume $X$ is a subset of $V(G)$ and $\phi$ is an $L$-colouring of $G[X]$. Let $L^\phi$ be the list assignment of $G-X$ defined as $$L^\phi(v)=L(v)-\{\phi(u):u\in N_X(v)\}$$ for each vertex $v\in V(G)-X$.
\end{definition}

The following easy observation (especially the "moreover" part) will be frequently used (often implicitly) in later arguments.

\begin{obs}
    \label{obs-lphi}
    Assume $X$ is a subset of $V(G)$ and $\phi$ is an $L$-colouring of $G[X]$. If $v \notin X$ and $|L(v)| = d_G(v)$, then 
    $|L^{\phi}(v)| \ge d_{G-X}(v)$. Moreover, if for some $u \in N_X(v)$, 
    $\phi(u) \notin L(v)$ or there are distinct vertices $u, u' \in N_X(v)$ with $\phi(u) = \phi(u')$, then 
    $|L^{\phi}(v)| > d_{G-X}(v)$.
\end{obs}

Assume  $Q$ is a connected component of $G[V_1]$. A \textit{leaf block} of $Q$ is a block $B$ of $Q$ that contains at most one cut-vertex of $Q$. If $B$ contains one cut-vertex $v$ of $Q$, then $v$ is called the {\em root} of $B$. The other vertices of $B$ are called {\em non-root}  vertices. If $Q$ is 2-connected, then  $Q$ itself is called a \textit{leaf block} and all vertices of $Q$ are non-root vertices. For a leaf-block $B$ of $Q$, we denote by $U_B$ the set of non-root vertices of $B$.

The following proposition is a variation and combination of  Lemma~3, Lemma~4 and Corollary~1 of \cite{JXXZ}.  
We sketch a proof of this proposition, and refer to \cite{JXXZ} for the detailed argument. 
\begin{proposition}\label{prop}
$G[V_2]$ is connected and for each connected component $Q$  of $G[V_1]$, the following hold:
\begin{itemize}
    \item[(1)] $Q$ is a Gallai tree.
    \item[(2)] Every block  of  $Q$ is either a complete graph $K_n$ with $n \le 3$ or an odd cycle.
    \item[(3)] For each non-cut vertex $v$ of $Q$, $d_Q(v) \le 2$  and is adjacent to some vertex in $V_2$.   
    \item[(4)] If $\theta_Q$ is a finite face of $G[V_2]$, then $\theta_Q=C_Q$ is a cycle.
    If $\theta_Q$ is an infinite face of $G[V_2]$ and $v^*$ is not an isolated vertex in $G[V_2]$, then each non-cut vertex $v$ of $Q$ satisfies $N_{\theta_Q}(v)\not=\{v^*\}$.
\end{itemize}
\end{proposition}
\begin{proof}
 Assume $G[V_2]$ has two connected parts $G_1$ and $G_2$.  Assume $v^* \in V(G_1)$, and $G_2$ is contained in face $\theta$ of $G_1$. If $G_2$ is contained in the interior (respectively, exterior) of $\theta$, then let $G'$ be the subgraph of $G$ induced by vertices contained in the interior of $\theta$ (respectively, exterior of $\theta$). 

By the minimality of $G$,  $G-G'$ has an $L$-colouring $\phi$. Then $L^{\phi}$ is a list assignment of $G'$ that satisfies the condition of Theorem \ref{thm-main2}. By the minimality of $G$, $G'$ has an $L^{\phi}$-colouring $\psi$. The union $\phi \cup \psi$ is an $L$-colouring of $G$. 

(1) If a connected component $Q$ of $G[V_1]$ is not a Gallai-tree, then by the minimality of $G$, $G-Q$ has an $L$-colouring $\phi$. Then $L^{\phi}$ is a degree-list assignment of $Q$ (see Observation \ref{obs-lphi}), and hence $Q$ has an $L^{\phi}$-colouring $\psi$.    The union $\phi \cup \psi$ is an $L$-colouring of $G$. 

(2) If a connected component $Q$ of $G[V_1]$ has a non-cut vertex $v$ that is not adjacent to any vertex of $V_2$, then let $\phi$ be an $L$-colouring of $v$. Then $L^{\phi}$ is a list assignment of $G-v$ that satisfies the condition of Theorem \ref{thm-main2}, and hence $G-v$ has an $L^{\phi}$-colouring $\psi$. The union $\phi \cup \psi$ is an $L$-colouring of $G$. Thus we may assume that every non-root vertex  of a leaf block of $Q$ is adjacent to a vertex of $V_2$. This implies that no block of $Q$ is a copy of $K_4$. Hence each block of $Q$ is an odd cycle and or a complete graph of order at most 3.

(3) follows from (2) and the assumption that $Q$ is properly connected to $V_2$.

(4) Assume $\theta_Q$ is a finite face and $\theta_Q$ is not a cycle. Let $C$ be a cycle in $V(\theta_Q)$. Then $C$ contains   a cut-vertex $u^*$ of $G[V_2]$.  By the minimality of $G$, $G-{\rm int}(C)$ has an $L$-colouring $\phi$. Let $G'$ be the subgraph of $G$ induced by ${\rm int}(C) \cup \{u^*\}$. Then 
$L^{\phi}$ is a list assignment of $G' \cup \{u^*\}$ that satisfies the condition of Theorem \ref{thm-main2}, with $u^*$ plays the role of $v^*$. Hence $G' \cup \{u^*\}$ has an $L^{\phi}$-colouring $\psi$. The union $\phi \cup \psi$ is an $L$-colouring of $G$.  

Assume that $\theta_Q$ is an infinite face of $G[V_2]$ and $v^*$ is not an isolated vertex in $G[V_2]$ and  there exists a non-cut vertex $v$ of $Q$ with $N_{\theta_Q}(v)=\{v^*\}$.
As $|L(v^*)|=1$, $L(v)-L(v^*) \ne \emptyset$. 
Let $\phi$ be a colouring of $v$ with a color $a\in L(v)\setminus L(v^*)$. Then $L^{\phi}$ is a list assignment of $G-v$ that satisfies the condition of Theorem \ref{thm-main2}, and hence has an $L^{\phi}$-colouring $\psi$. The union $\phi \cup \psi$ is an $L$-colouring of $G$.
\end{proof}

Let $F(G[V_2])$ be the set of faces of $G[V_2]$, and let $\Theta$ be the bipartite graph with partite sets $V_2$ and $F(G[V_2])$ in which $v\theta$ is an edge if and only if $v$ is a vertex on the boundary of $\theta$. We shall need the following lemma proved in \cite{ZZZ25}:

\begin{lem}
    \label{lem-protector}
    There is a spanning subgraph $H$ of $\Theta$ such that $d_H(v) \le 2$ for each vertex $v \in V_2$, $d_H(\theta) \ge d_{\Theta}(\theta)-2$ for each $\theta \in F(G[V_2])$. Moreover, if $|V_2| \ge 2$, then $d_H(v^*)=0$ and $d_H(\theta^*)=d_{\Theta}(\theta^*)-1$, if $V_2 = \{v^*\}$, then $H$ consists of a single edge $v^*\theta^*$ (in this case, $G[V_2]$ has a single face $\theta^*$).
\end{lem}

Assume $H$ is a spanning subgraph of $\Theta$ as described in Lemma \ref{lem-protector}.

\begin{definition}
    \label{def-protector}
    If $Q$ is a connected component of $G[V_1]$,  and $v\theta_Q \in E(H)$, then we say $v$ is a {\em protector} of $Q$. If $v\theta_Q \in E(\Theta)-E(H)$, then we say $v$ is a {\em non-protector} of $Q$. 
\end{definition}

It follows from Lemma \ref{lem-protector} that each vertex $v \in V_2$ is a protector of at most two connected components of $G[V_1]$, and each connected component $Q$ of $G[V_1]$ has at most 2 non-protectors. 

It follows from Observation \ref{obs-lphi} that for any $L$-colouring $\phi$ of $G[V_2]$, for any connected component $Q$ of $G[V_1]$, 
$|L^\phi(v)| \ge d_Q(v)$ for each vertex $v\in V(Q)$.
If $Q$ is not $L^\phi$-colourable, then we say $L^\phi$ is a {\em bad list assignment} for $Q$. 

The key idea in the proof of Theorem \ref{thm-main2} is that we find an $L$-colouring $\phi$ of $G[V_2]$ so that for each connected component $Q$ of $G[V_1]$,  $L^\phi$ is not a bad list assignment for $Q$.

The following was proved in \cite{ERT}.

\begin{lem} 
\label{lem-Gallai} 
If $Q$ is a   Gallai-tree, and $L^\phi$ is a bad list assignment for $Q$, then 
   for each block $B$ of $Q$ that is $r$-regular, there is a set $C_B$ of $r$ colours such that (i) if $B$ and $B'$ share a vertex, then $C_B \cap C_{B'} = \emptyset$, and (ii)  for each vertex $v$, $L^\phi(v)= \cup_{v \in B}C_B$. 
\end{lem}

\begin{cor}\label{cor-suff}
 Assume $L^\phi$ is a list assignment of a Gallai-tree $Q$. Then $L^\phi$ is not bad if one of the following holds:
\begin{enumerate}
    \item[(P1)] There exists a vertex $v \in V(Q)$ for which $|L^{\phi }(v)| > d_{Q}(v)$.
    \item[(P2)] There are two non-cut vertices $u,v$ of a same block of $Q$ with $L^{\phi}(u) \ne L^{\phi}(v)$.
\end{enumerate}   
\end{cor}

Our goal is to find an $L$-colouring $\phi$ of $G[V_2]$ so that for each connected component $Q$ of $G[V_1]$, (P1) or (P2) holds. This goal is achieved by restricting the colours assigned to protectors of $Q$.

\begin{definition}
    \label{def-sq}
    Assume $Q$ is a connected component of $G[V_1]$, and $\theta_Q$ is the face of $G[V_2]$ containing $Q$. 
    If $S_Q$ is an assignment of $V(\theta_Q)$ for which the following hold:
    \begin{enumerate}
        \item $|S_Q(v)| \le 3$ if $v$ is a protector of $Q$, and $S_Q(v)=\emptyset$ if $v$ is a non-protector of $Q$,
        \item if $\phi$ is an $L$-colouring   of $V(\theta_Q)$ such that $\phi(v) \notin S_Q(v)$ for each vertex $v\in V(\theta_Q)$, then there is an $L^\phi$-colouring of $Q$ (i.e., $L^\phi$ is not a bad list assignment for $Q$), 
    \end{enumerate}
    then we say $S_Q$ is a {\em valid assignment } for $\theta_Q$.
\end{definition}

\begin{lem}
    \label{lem-key}
    For each  connected component $Q$ of $G[V_1]$, $\theta_Q$ has a valid assignment.
\end{lem}

We leave the proof of Lemma \ref{lem-key} to the next section. Now we use this lemma to prove Theorem \ref{thm-main2}. For each connected component $Q$ of $G[V_1]$, let $S_Q$ be a valid assignment for $\theta_Q$.
For each vertex $v$ of $V_2$, let $$L'(v)=L(v)- \cup\{ S_{Q}(v): v \text{ is a protector of } Q\}.$$ 
For every vertex $v\in V_2-\{v^*\}$, since $v$ is a protector of at most two connected components of $G[V_1]$, $|L'(v)| \ge 11-6 = 5$. For the vertex $v^*$, either $|L'(v^*)|=|L(v^*)|= 1$ or $|L(v^*)| =11$ and $|L'(v^*)|\ge |L(v^*)|-3=8$. As $G[V_2]$ is a planar graph, by the classical result of Thomassen on list colouring of planar graphs, $G[V_2]$ has a proper $L'$-colouring. By the definition of valid assignment, $\phi$ can be extended to a proper $L$-colouring of $G$.

\section{Proof of  Lemma \ref{lem-key}} 

 Assume to the contrary that Lemma~\ref{lem-key} is false, and $Q$ is a connected component of $G[V_1]$ which has no valid assignment.

\begin{definition}
    \label{def-forced}
Assume $w$ is a vertex in $V(\theta_Q)$. We say $w$  is {\em confined   to colour $c \in L(w)$} if  for any proper $L$-colouring $\phi$ of $\theta_Q$ with $\phi(w) \ne c$, there is an $L^{\phi}$-colouring of $Q$. 
\end{definition}

\begin{lem}
    \label{lem-forcedc}
    Assume $w \in V(\theta_Q)$ is confined  to colour $c \in L(w)$. Then the following hold:
    \begin{enumerate}
        \item For any $v \in N_{Q}(w) $, $c \in L(v)$.
        \item $w$ is   a non-protector of $Q$.
    \end{enumerate}
\end{lem}
\begin{proof}
    (1) If $v \in N_{Q}(w) $ and $c \notin L(v)$, then  for any proper $L$-colouring $\phi$ of $\theta_Q$, either $\phi(w) \ne c$, or 
    $ \phi(w) = c \notin L(v)$ and hence $|L^{\phi}(v)| > d_{Q}(v)$. 
    By Corollary \ref{cor-suff}, there is an $L^\phi $-colouring of $Q$. So $S_Q(u)=\emptyset$ for each vertex $u \in V(\theta_Q)$ is a valid assignment for $\theta_Q$,  a contradiction.   

    (2)  If $w$ is a protector of $Q$ and is confined to colour $c$, then let $S_Q(w)=\{c\}$ and let $S_Q(u)=\emptyset$ for every other vertex $u\in V(\theta_Q)$. If $\phi$ is an $L$-colouring of $V(\theta)$ such that $\phi(u) \notin S_Q(u)$ for each vertex $u$, then by \Cref{def-forced}, there is an $L^\phi $-colouring of $Q$. So $S_Q$ is a valid assignment for $\theta_Q$,  a contradiction.
\end{proof}

 Let $W$ be the set of non-protectors of $Q$ in $V(\theta_Q)$,  and let $$F = \{w \in W: w\text{ is confined to some colour $c$}  \}.$$
 If $w$ is confined to colour $c$, then in the process of finding a valid assignment for $\theta_Q$,  we may treat $w$ is coloured by colour $c$.

\begin{lem}\label{lem-adjacentW}
 Let $B$ be a leaf block of $Q$ and $v \in U_{B}$. Then $v$ is adjacent to a non-protector of $Q$.
 Moreover, if $v$ is adjacent to a protector of $Q$, then
 $d_Q(v)+|N_{W\backslash F}(v)|\ge 3$.
  \end{lem}
\begin{proof}
By (3) of Proposition \ref{prop}, every vertex $v \in U_B$ is adjacent to some vertex in $V_2$. Assume $v$ is adjacent to a protector of $Q$. It suffices to prove that $d_Q(v)+|N_{W\backslash F}(v)|\ge 3$, which implies that $v$ is also adjacent to a non-protector (as $d_Q(v) \le 2$ by (3) of Proposition \ref{prop}).
 Assume to the contrary that $d_Q(v)+|N_{W\backslash F}(v)|\le 2$. Let $F'=N_F(v)$. 
For $w\in F'$, let $c_w$ be the colour such that $w$ is confined to colour $c_w$.  By Lemma \ref{lem-forcedc}, we have $$C_{F'}=\{c_w: w \in F'   \} \subseteq L(v).$$ 
Since $v$ is adjacent to some protector of $Q$, we conclude that  
$$|L(v)- C_{F'}|\geq d_G(v)-|N_{F'}(v)|\ge d_Q(v)+|N_{W\backslash F}(v)|+1.$$
Let $t=d_Q(v)+|N_{W\backslash F}(v)|+1$. As $d_Q(v)+|N_{W\backslash F}(v)|\le 2$, we have $t\leq 3$. Let $A$ be a $t$-subset of $L(v)\backslash C_{F'}$.
Let $S_Q(u) = A$ for each protector $u\in N_{\theta_Q}(v)$ of $Q$ and let $S_Q(u) = \emptyset$ for every other vertex $u\in V(\theta_Q)$.

For  any proper $L$-colouring $\phi$ of $\theta_Q$ with  $\phi(u) \notin S_Q(u)$, we have either $\phi(w)\ne c_w$ for some $w\in F'$ or $|L^\phi(v)|\ge t-|N_{W\backslash F}(v)|>d_Q(v)$. 
By Corollary \ref{cor-suff}, there is an $L^\phi$-colouring of $Q$. So $S_Q$ is a valid assignment for $Q$,  a contradiction.
\end{proof}

\begin{cor}
  If $B$ is a leaf block of $Q$ and  $v \in U_{B}$ is   adjacent to a protector of $Q$, then either $B$ is a copy of $K_2$, $F=\emptyset$ and $|W|=2$, and $v$ is adjacent to both vertices in $W$, or 
     $B$ is an odd cycle, and $v$ has a neighbor in $W \backslash F$.
     Consequently, $|V_2| \ge 2$.
\end{cor}
\begin{proof}
    By (2) of Proposition \ref{prop}, $B$ is a complete graph of order at most 3 or an odd cycle.
    If $B=K_1$, then $d_Q(v)=0$ and hence $d_Q(v)+|N_{W\backslash F}(v)|\le 2$, a contradiction.
    If $B=K_2$, then $d_Q(v)=1$, and hence $|N_{W\backslash F}(v)| = |W|=2$.  
    If $B$ is an odd cycle, then $d_Q(v)=2$, and hence $|N_{W\backslash F}(v)|\ge 1$. 
    As a consequence, $Q$ has a non-protector. But if  $|V_2|=1$, then $G[V_1]$ has a single component $Q$, which has no non-protector. So $|V_2|\ge 2$.
\end{proof}

\begin{lem}\label{lem-finite}
$\theta_Q$ is a finite face of $G[V_2]$.
\end{lem}
\begin{proof}
 Assume $\theta_Q$ is the infinite face of $G[V_2]$ and $v^*\in  V(\theta_Q)$. Thus, each vertex in $V(\theta_Q)-\{v^*\}$ is a protector of $Q$. Assume $B$ is a leaf block of $Q$ and $v\in U_B$ adjacent to $v^*$.  By  (4) and (4) of \Cref{prop}, we know that $N_{\theta_Q}(v)\not=\{v^*\}$, and hence $v$ is adjacent to some protector of $Q$. Assume $L(v^*) = \{c\}$. As $|L(v)|=d_G(v) \ge d_Q(v)+2$,  there is a  subset $A$ of $L(v) - \{c\}$ of size $d_Q(v)+1 \le 3$. Let $S_Q(u) =A$ for each protector $u\in N_{\theta_Q}(v)$ of $Q$ and $S_Q(u) = \emptyset$ for every other vertex $u\in V(\theta_Q)$. If $\phi$ is an $L$-colouring of $V(\theta_Q)$ such that $\phi(u) \notin S_Q(u)$ for each vertex $u$, then $|L^{\phi}(v)| > d_Q(v)$. Hence $L^{\phi}$ is not a bad assignment for $Q$. So  $S_Q$ is a valid assignment for $\theta_Q$, 
  a contradiction.
 \end{proof}

By  (4) of \Cref{prop} and Lemma \ref{lem-finite}, we conclude that $\theta_Q$ is a cycle. Without loss generality, we may assume that $W=\{w_1, w_2\}$ is the set of non-protectors of $Q$. Since $|V(\theta_Q)|\ge 3$, there is at least one protector of $Q$ which is adjacent to a vertex in $Q$ by \Cref{def-connected}.

\begin{lem}
     \label{lem-A}
     Assume that  $v_1, v_2\in U_{B}$ for a leaf block $B$ of $Q$. 
If  $N_{\theta_Q}(v_1) \subsetneq N_{\theta_Q}(v_2)$, then $N_{\theta_Q}(v_2) - N_{\theta_Q}(v_1) = \{w_i\}$ for some $i \in \{1,2\}$, and $w_i\in F$.
 \end{lem}
\begin{proof}
Assume to the contrary that $N_{\theta_Q}(v_2) - N_{\theta_Q}(v_1)=A\ne \{w_i\}$ for $i=1,2$. 
By  Lemma \ref{lem-adjacentW},  each of $v_1$ and $v_2$ is adjacent to at least one of $w_1, w_2$. As $N_{\theta_Q}(v_1) \subsetneq N_{\theta_Q}(v_2)$,  $A \ne \emptyset$ contains at most one non-protector of $Q$.
Since $v_1$ and $v_2$ are non-root vertices of a leaf block of $Q$, we have $d_Q(v_1)=d_Q(v_2)$.
Thus, $|L(v_2)-L(v_1)|\ge |L(v_2)|-|L(v_1)|=d_G(v_2)-d_G(v_1)=|N_{\theta_Q}(v_2)|-|N_{\theta_Q}(v_1)|=|A|$.

If  each vertex in $A$ is a protector of $Q$, then let
 $a\in L(v_2)-L(v_1)$ and let $S_Q(u)=\{a\}$  for each protector $u\in N_{\theta_Q}(v_2)$ of $Q$ and  $S_Q(u)=\emptyset$ for every other vertex $u\in V(\theta_Q)$.
Assume $\phi$ is an  $L$-colouring  of $\theta_Q$ such that  $\phi(u) \notin S_Q(u)$ for each vertex $u\in V(\theta_Q)$. Then   
 either $a\in L^\phi(v_2)-L^\phi(v_1)$ and hence $L^\phi(v_1)\ne L^\phi(v_2)$, or there exists a vertex $u'\in N_{\theta_Q}(v_1)\cap N_{\theta_Q}(v_2)$ such that $\phi(u')=a$ and hence $|L^\phi(v_1)|>d_Q(v_1)$.
By Corollary \ref{cor-suff}, there is  an $L^\phi$-colouring of $Q$. So $S_Q$ is a valid assignment for $\theta_Q$, a contradiction.

Assume exactly one of $w_1$ and $w_2$ is contained in $A$, say $ A \cap W =\{w_1\}$. By  Lemma \ref{lem-adjacentW}, each of $v_1$ and $v_2$ is adjacent to at least one of $w_1, w_2$. Since $w_1\notin N_{\theta_Q}(v_1)$, we know that $w_2\in N_{\theta_Q}(v_1)$. As $N_{\theta_Q}(v_1) \subsetneq N_{\theta_Q}(v_2)$,   $w_2\in N_{\theta_Q}(v_1)\cap N_{\theta_Q}(v_2)$.  
 Assume $A\ne \{w_1\}$. Then  $A$ contains a protector of $Q$. 
Let $\{a,b\}$ be a 2-subset of $L(v_2)-L(v_1)$, and let $S_Q(u) = \{a,b\}$ for each protector $u\in N_{\theta_Q}(v_2)$ of $Q$ and let $S_Q(u)=\emptyset$ for every other vertex $u\in V(\theta_Q)$.

Assume $\phi$ is an $L$-colouring  of $\theta_Q$ such that  $\phi(u) \notin S_Q(u)$ for each vertex $u\in V(\theta_Q)$. Then  
at least one of the following holds:
\begin{itemize}
 
    \item $\phi(w_2)\in \{a,b\}$, which implies $L^\phi(v_1)>d_Q(v_1)$.

    \item $\phi(w_2)\notin \{a,b\}$, which implies $\{a,b\}-\phi(w_1)\subseteq L^\phi(v_2)$ and hence $L^\phi(v_2)\ne L^\phi(v_1)$.
\end{itemize} 
By Corollary \ref{cor-suff}, there is  an $L^\phi$-colouring of $Q$. So  $S_Q$ is a valid assignment for $\theta_Q$, a contradiction.

Assume $A= \{w_1\}$. We shall prove that $w_1 \in F$.

Let  $c\in L(v_2)-L(v_1)$. Assume $\phi$ is a proper $L$-colouring of $Q$ with  $\phi(w_1) \neq c$. If $c \in L^{\phi}(v_2)$, then $L^{\phi}(v_1) \ne L^{\phi}(v_2)$. If $c \notin L^{\phi}(v_2)$, then a common neighbour of $v_1$ and $v_2$ is coloured by $c$, and hence $|L^{\phi}(v_1)| > d_{Q}(v_1)$. In both cases, it follows Corollary \ref{cor-suff} that there is an $L^\phi$-colouring of $Q$. Therefore, $w_1$ is confined to the colour $c$, and hence $w_1\in F$.
\end{proof}

\begin{lem}\label{lem-oddcycle}
Assume  a leaf block $B$ of $Q$ is an odd cycle, 
 $v_B \in U_B$ is adjacent to some protector of $Q$. Then for any vertex $v \in U_B\backslash \{v_B\}$,  the following hold:
\begin{enumerate}
\item If $N_{W}(v)\subseteq N_{W}(v_B)$, then $v$ is adjacent to some protector of $Q$.

\item If $N_{W}(v_B)=N_{W}(v)=\{w_i\}$ for some $i=1, 2$, then $N_{\theta_Q}(v_B)\cap N_{\theta_Q}(v)$ contains a protector of $Q$.

\item If $N_{W}(v_B)\cap N_{W}(v)=\emptyset $, then $|L(v_B)\cap L(v)|\ge 3$.

\end{enumerate}
\end{lem}
\begin{proof}
(1) 
If $v$ is not adjacent to any protector of $Q$, then
$N_{\theta_Q}(v)\subsetneq N_{\theta_Q}(v_B)$.
But $N_{\theta_Q}(v_B)-N_{\theta_Q}(v)$ contains at least one  protector of $Q$, in contrary to Lemma \ref{lem-A}.

(2)
It follows from (1) that $v$ is adjacent to at least one protector of $Q$.
Thus, $|L(v_B)|, |L(v)|\ge 4$.
Assume to the contrary that $N_{\theta_Q}(v_B)\cap N_{\theta_Q}(v)$ does not contain any protector of $Q$. 
We may assume that $N_{\theta_Q}(v_B)\cap N_{\theta_Q}(v)=\{w_1\}$.
Let $\{a_1,a_2,a_3\}\subseteq L(v_B)$ and $\{b_1,b_2,b_3\}\subseteq L(v)$, 
where $\{a_1,a_2,a_3\}\ne \{b_1,b_2,b_3\}$. 
 Define
\[
S_Q(u)=
\begin{cases}
\{a_1,a_2,a_3\}, & \text{if } u\in N_{\theta_Q}(v_B)- \{w_1\},\\[4pt]
\{b_1,b_2,b_3\}, & \text{if } u\in N_{\theta_Q}(v)- \{w_1\},\\[4pt]
\emptyset, & \text{otherwise.}
\end{cases}
\]

Since $ N_{\theta_Q}(v_B)\cap N_{\theta_Q}(v)=\{w_1\}$, we have $|S_Q(u)|\le 3$ for each protector $u$ of $Q$.
Assume $\phi$ is an $L$-colouring  of $\theta_Q$ such that  $\phi(u) \notin S_Q(u)$ for each vertex $u\in V(\theta_Q)$. Then  $\{a_1,a_2,a_3\}- \{\phi (w_1)\}\ne \{b_1,b_2,b_3\}- \{\phi (w_1)\}$, and hence $L^\phi(v_B)\neq L^\phi(v)$.
By  (P2) of Corollary \ref{cor-suff}, there is an $L^\phi$-colouring of $Q$. So $S_Q$ is a valid assignment for $\theta_Q$, a contradiction.

(3)  Assume to the contrary that $|L(v_B)\cap L(v)|\le 2$. By  (3) of \Cref{prop} and \Cref{lem-adjacentW}, each of $v_B$ and $v$ is adjacent to some vertex in $W\backslash F$. Without loss of generality, we may assume that $N_{W}(v_B)=\{w_1\}$ and $ N_{W}(v)=\{w_2\}$. Since $|L(v_B)|=d_G(v_B)\ge 4$, we conclude that $|L(v_B)-L(v)|\ge 2$.
Let $\{a,b\}$ be a 2-subset of $ L(v_B)-L(v)$, and let $S_Q(u)=\{a,b\}$ for every protector $u\in N_{\theta_Q}(v_B)$ and let $S_Q(u)=\emptyset$ for every other vertex $u\in V(\theta_Q)\backslash N_{\theta_Q}(v_B)$. Assume $\phi$ is an  $L$-colouring  of $\theta_Q$ such that  $\phi(u) \notin S_Q(u)$ for each vertex $u\in V(\theta_Q)$.
Then $\{a,b\}-\{\phi(w_1)\} \subseteq L^\phi(v_B)-L^\phi(v)$, and hence $L^\phi(v_B)\ne L^\phi(v).$
 By  (P2) of Corollary \ref{cor-suff}, there is an $L^\phi$-colouring of $Q$. So $S_Q$ is a valid assignment for $\theta_Q$, a contradiction. 
\end{proof}

\begin{lem}\label{w1,w2}
If $B$ is a leaf block of $Q$, then $w_1,w_2\in N_{\theta_Q}(U_B)$.
\end{lem}
\begin{proof}
Assume to the contrary that $w_2\notin N_{\theta_Q}(U_B)$.
Since $Q$ is properly connected to $V_2$, we know that $|N_{\theta_Q}(U_B)|\ge 2$. Hence  $N_{\theta_Q}(U_B)$ contains a protector of $Q$. 
By Lemma \ref{lem-adjacentW}, for each $v\in U_B$, $v$ is adjacent to $w_1$ and $d_Q(v)=2$. 
Thus, $B$ is an odd cycle $C=[v_1,v_2,\dots,v_{2l+1}]$ with $l\ge1$.
Without loss of generality, we assume that $v_1,v_2\in U_B$ and $v_1$ is adjacent to some protector of $Q$.
By (2) of \Cref{lem-oddcycle} , $ N_{\theta_Q}(v_1)\cap N_{\theta_Q}(v_2)$ contains a protector $u$  of $Q$.
Thus, $Q-\{v_1,v_2\}$ is contained in the interior of the $4$-cycle $[v_1uv_2w_1]$.
If $Q$ has another leaf block $B'$, then $|N_{\theta_Q}(B')|\le 1$, again  contrary to the assumption that $Q$ is properly connected to $V_2$.
Thus, $Q$ is an odd cycle. 
By Lemma \ref{lem-adjacentW},   $v_{2l+1}$ is adjacent to a non-protector of $Q$. By planarity, $N_{\theta_Q}(v_{2l+1})=\{w_1\}$, in contrary to (1) of \Cref{lem-oddcycle} .
\end{proof}

The following   corollary follows from Lemma \ref{w1,w2} and the planarity of $G$.

\begin{cor}
    \label{2leaf}
  $Q$ has at most two leaf-blocks, and each protector $u$ of $Q$  is adjacent to non-root vertices of exactly one leaf-block of $Q$. 
\end{cor}
\begin{proof}
    If $Q$ has three leaf-blocks, then $Q$ contains $K_{1,3}$ as a minor and by Lemma \ref{w1,w2}, $Q \cup W$ contains $K_{3,3}$ as a minor, a contradiction.

    Assume $Q$ has two leaf-blocks $B$ and $B'$, and  $u$ is a protector of $Q$ adjacent to both $U_B$ and $U_{B'}$. 
    By contracting  the cycle $\theta_Q$ into a triangle containing $w_1,w_2,u$, and contracting each of $U_B$ and $U_{B'}$ into a single vertex, we obtain  a copy of $K_5$,   a contradiction.   
\end{proof}

\begin{lem}\label{lem-2blocks}
    $Q$ has exactly two leaf blocks.
\end{lem} 
\begin{proof}
Suppose not, then by Corollary \ref{2leaf}, $Q$ is a complete graph $K_n$ ($n\le3$) or an odd cycle. It follows \Cref{lem-adjacentW} that $Q$ is not a copy of $K_{1}$.

If $Q$ is $K_{2}$ with vertices $v_1,v_2$, then by symmetry, we assume $v_1$ is adjacent to some protector of $Q$.
By Lemma \ref{lem-adjacentW}, we conclude that $v_1$ is adjacent to $w_1$ and $w_2$. 
By (1) of \Cref{lem-oddcycle} ,  $v_2$ is  adjacent to some protector of $Q$. 
Again by \Cref{lem-adjacentW}, we conclude that $v_2$ is also adjacent to $w_1$ and $w_2$. Thus, $N_{\theta_Q}(v_1)\cap N_{\theta_Q}(v_2)=W$ (for otherwise, $G$ contains $K_5$ as a minor). Now we know that $|L(v_1)|, |L(v_2)|\ge 4$.

Let $\{1,2,3\}\subseteq L(v_1)$ and $b\in  L(v_2)- \{1,2,3\}$.
Define
\[
S_Q(u)=
\begin{cases}
\{1,2,3\}, & \text{if } u\in N_{\theta_Q}(v_1)- W,\\[4pt]
\{b\}, & \text{if } u\in N_{\theta_Q}(v_2)- W,\\[4pt]
\emptyset, & \text{otherwise.}
\end{cases}
\]

Since $ N_{\theta_Q}(v_1)\cap N_{\theta_Q}(v_2)=W$, $S_Q$ is well-defined.
For  any proper $L$-colouring $\phi$ of $\theta_Q$ such that  $\phi(u) \notin S_Q(u)$ for each vertex $u\in V(\theta_Q)$, we have $\{1,2,3\}- \{\phi (w_1),\phi (w_2)\}\subseteq L^\phi (v_1)$. 
If $|L^\phi(v_1)| \ge 2 > d_Q(v_1)=1$, then by Corollary \ref{cor-suff}, there is an $L^\phi$-colouring of $Q$. Otherwise, $\{\phi(w_1), \phi(w_2) \} \subseteq \{1,2\}$, and $b \in L^\phi(v_2) - L^\phi(v_1)$. Again by Corollary \ref{cor-suff}, there is an $L^\phi$-colouring of $Q$. So $S_Q$ is a valid assignment for $\theta_Q$,  a contradiction.

Now assume that $Q$ is an odd cycle $C=[v_1v_2\ldots v_{2l+1}]$ with $l\ge1$. 
For each vertex $w \in \{ w_1, w_2\}$, $N_{Q}(w)$ is a subpath of $C$.
By Lemma \ref{lem-adjacentW}, each  $v_i$ is adjacent to at least one of $w_1,w_2$ for each $i\in \{1,2,\dots,2l+1\}$.

If there exist two vertices, say $v_1, v_j$, in $V(Q)$ that are adjacent to both $w_1$ and $w_2$, then  $Q - \{v_1, v_j\}$ is contained in the $4$-cycle $[v_1w_1v_jw_2]$. 
Hence, $N_{\theta_Q}(v_i)\in \big \{\{w_1\},\{w_2\}\big\}$ for $i\ne 1,j$. 
Then $v_1$ or $v_j$ is adjacent to some protector of $Q$, and this is in contrary to (1) of \Cref{lem-oddcycle} .
So at most one vertex of $C$ is adjacent to both $w_1,w_2$.

\medskip
{\bf Case 1} There is a vertex, say $v_1$,  in $V(Q)$ that is adjacent to both $w_1$ and $w_2$. 

\medskip

Since $N_Q(x)$ is a subpath of $C$ for $x\in W$, we may assume that $v_2,\dots v_i~(i\ge 2)$ is adjacent to $w_1$ and every other vertex in $V(Q)$ is adjacent to $w_2$.

If $i=2l+1$, then by symmetry we assume  $Q-\{v_1,v_{2l+1}\}$ is contained in a $3$-cycle $[v_1v_{2l+1}w_1]$.
It implies that $v_j$ is not adjacent to any protector of $Q$ for $j\notin\{1,2l+1\} $.
If $v_{2l+1}$ is adjacent to some protector of $Q$, then by \Cref{lem-oddcycle},  $N_{\theta_Q}(v_2)\cap N_{\theta_Q}(v_{2l+1})$ contains at least one protector of $Q$, a contradiction. 
Therefore, each protector of $Q$ is adjacent to $v_1$, in contrary to (1) of \Cref{lem-oddcycle} .
Therefore, $i<2l+1$.
By the planarity of $G$, $v$ is not adjacent to any protector of $Q$ for $v\in V(Q)-\{v_1,v_i,v_{i+1}\}$.

If $v_i$ and $v_{i+1}$ are not adjacent to any protector of $Q$, then 
$v_1$ is adjacent to some protector of $Q$. But then  by (1) of \Cref{lem-oddcycle} , $v_i$ and $v_{i+1}$ are adjacent to some protector of $Q$, a contradiction.
By symmetry, we assume that $v_i$ is adjacent to some protector of $Q$.
By \Cref{lem-oddcycle} (3), $|L(v_i)\cap L(v_{i+1})|\ge 3$.

Let  $S_Q(u)=\{1,2,3\}\subseteq L(v_i)\cap L(v_{i+1})$ for each protector $u\in N_{\theta_Q}(v_i)\cup N_{\theta_Q}(v_{i+1})$ of $Q$ and let $S_Q(u)=\emptyset$ for every other vertex  $u\in V(\theta_Q)$.
For  any proper $L$-colouring $\phi$ of $\theta_Q$ such that  $\phi(u) \notin S_Q(u)$ for each vertex $u\in V(\theta_Q)$,  we have $\{1,2,3\}- \{\phi (w_1)\} \subseteq L^\phi (v_i)$  and $\{1,2,3\}- \{\phi (w_2)\}\subseteq L^\phi (v_{i+1})$.
Thus, at least one of the following holds:
\begin{itemize}
    \item $\phi (w_1)=\phi (w_2)$ and hence $L^\phi(v_1)>d_Q(v_1)$.
    \item $\phi(w_1)\notin \{1,2,3\}$ and hence $|L^{\phi}(v_i)|>d_{Q}(v_i)$ or $\phi(w_2)\notin \{1,2,3\}$ and hence $|L^{\phi}(v_{i+1})|>d_{Q}(v_{i+1})$.
    \item $\phi (w_1)\ne \phi (w_2)$ and $\phi(w_1),\phi (w_2)\in \{1,2,3\}$. Thus, $ L^{\phi}(v_i)\ne  L^{\phi}(v_{i+1})$.
\end{itemize} 
By Corollary \ref{cor-suff}, there is an $L^\phi$-colouring of $Q$. So $S_Q$ is a valid assignment for $\theta_Q$, a contradiction.

\medskip
{\bf Case 2} Each vertex in $V(Q)$ is adjacent to exactly one of $w_1,w_2$.
\medskip

By the pigeonhole principle, it follows that  $|N_Q(w_1)|\ge2$ or $|N_Q(w_2)|\ge2$.
If $|N_Q(w_1)|\ge2$ and there is a vertex  $ v_i\in N_Q(w_1)$ that is adjacent to some protector of $Q$, then by (1) of \Cref{lem-oddcycle} , there is a protector $u\in N_{\theta_Q}(v_i)\cap N_{\theta_Q}(v_j)$ of $Q$ for $v_i, v_j\in N_Q(w_1)$.
Thus, $Q-\{v_i,v_j\}$ is contained in a $4$-cycle $[w_1v_iuv_j]$. Hence, $N_Q(w_2)=\emptyset$, in contrary to Lemma \ref{w1,w2}. 

Therefore, either $|N_Q(w_1)|=1$, or every vertex in $N_Q(w_1)$ is not adjacent to any protector of $Q$. Similarly, the same holds for $w_2$. Hence, we may assume that $N_Q(w_1)=\{v_1\}$ and every protector of $Q$ is adjacent to $v_1$.
This contradicts to the assumption that $Q$ is properly connected to $V_2$.
\end{proof}

Assume the blocks of $Q$  are ordered as $B_1,B_2, \ldots, B_k$ ($k \ge 2$) such that $B_1$ and $B_k$ are leaf-blocks and for $i=1,2, \ldots, k-1$, $B_i$ and $B_{i+1}$ share a cut-vertex $z_i$, and $z_1, z_2, \ldots, z_{k-1}$ are pairwise distinct.  

By Lemma~\ref{w1,w2}, we have $w_1, w_2 \in U_{B_i}$ for $i = 1, k$. 
Thus the subgraph of $G$ induced by $W \cup U_{B_1} \cup U_{B_k}$ contains a cycle $C''$ such that 
$Q - U_{B_1}  - U_{B_k}$ is contained in the interior of $C''$, and each  protector of $Q$ is contained in the exterior of $C''$. Thus, no protector of $Q$ is adjacent to any vertex in $Q - U_{B_1}  - U_{B_k}$.
Moreover, by Corollary \ref{2leaf}, we conclude that  each protector of $Q$ is adjacent to non-root vertices of exactly one leaf-block of $Q$.
Since each vertex in $\theta_Q$ is adjacent to some  vertex in $Q$, 
 the protectors of $Q$ can be divided into two disjoint parts, $C^1$ and $C^k$, where 
$N_{\theta_Q}(U_{B_1})=C^1\cup W$ and $N_{\theta_Q}(U_{B_k})=C^k\cup W$,
and each $G[C^i]$ ($i=1,k$) is a  path on $\theta_Q$ or an empty graph.

\begin{lem}\label{lem-|N(wi)|=1}
    Each of $w_1$ and $w_2$ has exactly one neighbor in $U_{B_i}$ for $i=1, k$. 
\end{lem}
\begin{proof}
    Assume to the contrary that $w_1$ has at least two neighbours in $U_{B_1}$. Note that $B_1$ is an odd cycle in this case.  Let $x_1, y_1$ be the such two neighbours with $x_1y_1\in E(Q)$. Such vertices $x_1, y_1$ exist as $d_{G}(v)\ge 3$ for each vertex $v\in V(Q)$. Since $w_1$ is adjacent to some vertex in $U_{B_k}$, by the planarity, one of $x_1$ and $y_1$ has no other neighbour in $\theta_Q$, say $y_1$. Thus we have $N_{\theta_Q}(y_1) \subsetneq N_{\theta_Q}(x_1)$. It follows \Cref{lem-A} that $N_{\theta_Q}(x_1) - N_{\theta_Q}(y_1) = \{w_2\}$  and $w_2\in F$. 
    If $C^1\ne \emptyset$, then by the planarity, each vertex in $C^1$ is adjacent to $x_1$. It implies that  $x_1$ is adjacent to some protector of $Q$.  By (1) of \Cref{lem-oddcycle} , $y_1$ is also adjacent to some protector of $Q$, a contradiction. So $C^1= \emptyset$.

    By \Cref{lem-adjacentW},  each vertex in $U_{B_k}$ must be adjacent to the vertex in $\{w_1\}=W\backslash F$, and $B_k$ is an odd cycle as $F\neq \emptyset$. Now the vertex $w_1$ has at least two neighbours in $U_{B_k}$. We can find a pair $x_{k}, y_{k}$ in $U_{B_k}$, which plays the same role as the pair  $x_1, y_1$ in $U_{B_1}$. Using the same argument, we can get $C^k= \emptyset$. However, $|V(\theta_Q)|=|V(C^1)|+|V(C^k)|+|W|=2$, a contradiction. 
\end{proof}

\begin{lem}\label{lem-3}
Both $B_1$ and $B_k$ are complete graphs. Furthermore, if $B_i$ is a copy of $K_3$, then $C^i\ne \emptyset$.

\end{lem}
\begin{proof}
By Lemma \ref{lem-adjacentW}, for $i=1,k$, each vertex in $U_{B_i}$ is adjacent to some vertex in $W$.
Thus it follows from \Cref{lem-|N(wi)|=1} that $|U_{B_i}|\leq 2$, and hence $B_i$ is a complete graph. 

Assume that $B_1=[z_1x_1y_1]$ and $C^1=\emptyset$.  Then $w_1w_2 \in E(G)$. We may assume that $ x_1w_1, y_1w_2\in E(G)$. By \Cref{lem-oddcycle} (3), $|L(x_1) \cap L(y_1)| \ge 3$, and hence  $L(x_1)=L(y_1)$.  For every proper $L$-coloring $\phi$ of $\theta_Q$, $\phi(w_1)\neq \phi(w_2)$ as $w_1w_2\in E(G)$. So $L(x_1)\backslash\phi(w_1)=L(y_1)\backslash\phi(w_2)$. Thus $L^{\phi}(x_1)\neq L^{\phi}(y_1)$, which contradicts  (P2) of \Cref{cor-suff}.
\end{proof}

By symmetry, we assume that $C^k\not=\emptyset$. The leaf block $B_1$ is a copy of $K_2$ or $K_3$. In the following proof, if $B_1$ is a copy of $K_2$, then let $U_{B_1}=\{x_1\}$; if $B_1$ is a copy of $K_3$, we may assume that $U_{B_1}=\{x_1,y_1\}$ and $x_1$ is adjacent to $w_1$ and some protector of $Q$, and $y_1$ is adjacent to $w_2$ (\Cref{lem-3}).

If $B_1=[x_1y_1z_1]$ is a triangle with $\{x_1, y_1\} = U_{B_1}$, then by \Cref{lem-oddcycle} (3),  $|L(x_1)\cap L(y_1)|\ge 3$. 
 If $B_1=x_1z_1$ is a copy of $K_2$ with $\{x_1\} = U_{B_1}$, then  $x_1$ is adjacent to both $w_1$ and $w_2$. In any case, $\cap{ v \in U_{B_1}} L(v)$ contains a  set  of three colours, say $\{a_1,a_2,a_3\} \subseteq \cap{ v \in U_{B_1}} L(v)$. 

Let $S_Q(u)=\{a_1,a_2,a_3\}$ for each vertex $u\in C^1$ (note that $C^1$ could be empty). As $C^k\ne\emptyset$, let $x_k$ be the vertex in $U_{B_k}$ such that $x_k$ is adjacent to some protector of $Q$.
If $B_k$ is a copy of $K_2$, then $x_k$ is also adjacent to $w_1$ and $w_2$. If $B_k$ is a copy of $K_3$, then $x_k$ is adjacent $w_1$ or $w_2$. Thus we have $|L(x_k)|\ge 4$, and hence there is a colour $b\in L(x_k)-\{a_1,a_2,a_3\}$.

Let $L'(v)=\{a_1,a_2,a_3\}$ for each vertex $v\in U_{B_1}$ and let $L'(v)=L(v)$ for each vertex $v\in V(Q)-U_{B_1}$.
Let $l$ be the largest index such that $B_l$ has a vertex $x$ such that $b\notin  L'(x)$. 
Since  $b\notin L'(v)$ for each vertex $v\in U_{B_1}$, the index $l$ is well-defined. 
Let $\{c_1, c_2, c_3\}$ be a set of three colours contained in $L(u_k)-\{b\}$. For each protector $u\in N_{\theta_Q}(x_k)$, let 
\[
S_{Q}(u) = \begin{cases} \{b\}, &\text{ if $k-l$ is even}, \cr 
\{c_1,c_2,c_3\}, &\text{ if $k-l$ is odd}.
\end{cases}
\]
Let $S_Q(u)=\emptyset$ for every other vertex $u\in V(\theta_Q)-N_{\theta_Q}(u_k)-C^1$.

By our assumption, $S_Q$ is not a valid assginment for $\theta_Q$. Thus there exists  a proper $L$-colouring $\phi$ of $\theta_Q$ such that $\phi(u) \notin S_Q(u)$ for each vertex $u\in V(\theta_Q)$, and 
$L^\phi$ is a bad list assignment of $Q$. 
By Lemma \ref{lem-Gallai}, there is a colour set $C_i$ for each block $B_i$ with $1 \le i \le k$, such that for each vertex $x$ of $Q$, $L^{\phi}(x) = \cup_{x \in B_i}C_i$ and $C_i \cap C_{i+1} = \emptyset$. 

Assume $x_1 \in U_{B_1}$. 
As $S_Q(u)=\{a_1,a_2,a_3\}\subseteq \cap_{v\in U_{B_1}}L(v)$ for each vertex $u\in C^1$, we have $\{a_1, a_2, a_3\}\backslash\{\phi(w_1), \phi(w_2)\} \subseteq L^{\phi}(x_1)= C_1$.

If $B_1=x_1z_1$, then  $|C_1|=|L^\phi(x_1)| = d_{Q}(x_1)=1$, and hence $C_1\subseteq \{a_1,a_2,a_3\}$.

If $B_1=[x_1y_1z_1]$ and $N_W(x_1)=\{w_1\}$, then $\{a_1, a_2, a_3\}\backslash\{\phi(w_1)\} \subseteq L^{\phi}(x_1)= C_1$. As $|C_1|=|L^\phi(x_1)| = d_{Q}(x_1)=2$, we have $C_1\subseteq \{a_1,a_2,a_3\}$.

By the definition of the index $l$ and the colour $b$, we know that  $b\in L'(x)$ for every vertex $x\in V(B_{j})$ for each $j\in\{l+1,...,k\}$, where $l\leq k-1$. Since each protector of $Q$ is not adjacent to any vertex in $Q - U_{B_1}  - U_{B_k}$, for each vertex $x$ of $Q - U_{B_1}  - U_{B_k}$, we have $$L(x)  - \{\phi(w_1), \phi(w_2)\} \subseteq L^{\phi}(x).$$
In particular, $b \in L^{\phi}(x)$ if and only if $b\in L(x)=L'(x)$.

The colour $b$ is not contained in $C_l$ as $b\notin L'(x)$ for some vertex $x\in B_{l}$. Since $b \in L^{\phi}(x)$ for all $x \in B_{l+1}$, we know that $b \in C_{l+1}$.
 As $C_{l+1} \cap C_{l+2} = \emptyset$, we know that $b \notin C_{l+2}$ (note that it is possible when $B_{l+2}$ is a copy of $K_2$). Now $b \in L^{\phi}(x)$ for all $x \in B_{l+3}$. This implies that $b \in C_{l+3}$.  Repeat this argument, we conclude that   $b \in C_{l+1}, C_{l+3}, \ldots, C_{l+1+2t}, \ldots$.  

If $k-l$ is odd, then this implies that $b \in C_k$, and hence $\phi(u)\ne b$ for each protector $u\in N_{\theta_Q}(u_k)$ of $Q$. Therefore,  $\phi(u) \notin \{c_1,c_2,c_3,b\}$. Thus $|L^{\phi}(u_k)|\ge \{c_1,c_2,c_3,b\}\backslash \{\phi(w_i): w_i\in N_{G}(u_k)\} > d_{Q}(u_k)$, which implies $L^{\phi}$ is not a bad list assignment for $Q$, a contradiction.

If $k-l$ is even, then $b \notin C_l, C_{l+2}, \ldots, C_k$. But $\phi(u) \notin S_Q(u) = \{b\}$ for each protector $u\in N_{\theta_Q}(u_k)$ of $Q$ and $\phi(w_i) \ne b$ for $i=1,2$. Hence $b\in C_{k}$, a contradiction. \qed

\section{Concluding remark}

We have shown that every 3-connected non-complete planar graph is degree-truncated $11$-choosable, and we conjecture that this upper bound can be further reduced by $1$. 

The  proof of Theorem \ref{thm-main2} consists of two main steps. Let $V_1=\{v: d_G(v) \le 10\}$ and $V_2=V(G)-V_1$. 

The first step is that there is a way to assign protectors from $V_2$ to connected components $Q$ of $G[V_1]$ such that each vertex $v \in V_2$ is the protector of at most two components of $G[V_1]$, and each connected component $Q$ has at most two non-protectors among its neighbours in $V_2$. (This step was proved in \cite{JXXZ}). 

The second step is to find a valid assignment $S_Q$ for each connected component $Q$ of $G[V_1]$ so that $|S_Q(v)| \le 3$ for each protector $v$ of $Q$ (and $S_Q(v)=\emptyset$ if $v$ is not a protector of $Q$). The set $S_Q(v)$ is the "cost" for $v$ to protect $Q$. If there is an $L$-colouring of $G[V_2]$ such that $v$ does not use colours from $S_Q(v)$ for the components $Q$ that are protected by $v$, then this colouring can be extended to an $L$-colouring of $G$. This is so because each connected component $Q$ of $G[V_1]$ is "protected". 

In our proof, because $|L(v)| =11$, after deleting colours from $S_Q(v)$, each vertex $v$ still has at least 5 permissible colours. Hence there is an $L$-colouring of $G[V_2]$ such that $v$ does not use colours from $S_Q(v)$ for the components $Q$ that are protected by $v$. If $|L(v)|=10$ for each vertex $v \in V_2$, and $G[V_2]$ is 4-choosable, then such an $L$-colouring of $G[V_2]$ also exists. Thus if $G[V_2]$ is 4-choosable, then $G$ is degree-truncated $10$-choosable.

If $G[V_2]$ is not 4-choosable, then the required $L$-colouring of $G[V_2]$ may not exist. One possible approach is to use a different assignment of protectors for components of $G[V_1]$.  Also in some cases, the choices of $S_Q(v)$ have some flexibility, which we may explore. Another possible approach is instead of protecting components of $G[V_1]$ one by one, one may try to protect some components of $G[V_1]$ together with joint efforts from their neighbours. To prove Conjecture \ref{conj-1}, it seems that one needs to prove that every planar graph is in some sense "almost" 4-choosable. Conversely, a counterexample to Conjecture \ref{conj-1} would has a subgraph which is planar and in some sense  "strongly" non-4-choosable. This intuition suggests that to prove or disprove Conjecture \ref{conj-1} will not be easy.

\bibliographystyle{plain}
\bibliography{template}	
	 		
\end{document}